\documentclass[12pt,reqno,oneside]{amsart}
\usepackage{parskip}
\usepackage[]{babel}
\usepackage[utf8]{inputenc} 
\usepackage{amsmath,amsthm,amsbsy}
\usepackage{amsfonts}
\usepackage{yfonts}
\usepackage{mathrsfs}
\usepackage{bbm}
\usepackage{amssymb}
\usepackage[colorlinks=true, menucolor=blue, linkcolor=blue, citecolor=blue]{hyperref}
\usepackage{enumitem}
\usepackage{multicol}
\usepackage{float}
\usepackage{fancyhdr}
\usepackage{layout}
\usepackage{esint}
\usepackage{array}
\usepackage{makecell}
\allowdisplaybreaks

\usepackage[colorinlistoftodos]{todonotes}
\usepackage{orcidlink} 

\usepackage{mathtools}
\usepackage[foot]{amsaddr}

\theoremstyle{plain}
\newtheorem{theorem}{Theorem}[section]                                          
\newtheorem{proposition}[theorem]{Proposition}                          
\newtheorem{lemma}[theorem]{Lemma}

\theoremstyle{definition}
\newtheorem{definition}[theorem]{Definition}
\theoremstyle{remark}
\newtheorem{remark}[theorem]{Remark}
\newtheorem{example}[theorem]{Example}
\newtheorem{notation}[theorem]{Notation}
\newtheorem{assumption}[theorem]{Assumption}

\makeatletter \@addtoreset{equation}{section} \makeatother

\newcommand{\caract}{\mathbbm{1}}
\newcommand{\calF}{\mathcal{F}}
\newcommand{\calA}{\mathcal{A}}

\newcommand{\calB}{\mathcal{B}}

\newcommand{\calP}{\mathcal{P}}

\newcommand{\defeq}{:=}

\newcommand{\N}{\mathbb{N}}     
\newcommand{\R}{\mathbb{R}}     
\newcommand{\Prob}{\mathbb{P}}  
\newcommand{\Exp}{\mathbb{E}}   
\newcommand{\ind}[2]{\mathbbm{1}_{#1}\left( #2 \right)}          
\newcommand{\inner}[2]{\left( #1 \, , \, #2 \right)} 
\newcommand{\norm}[1]{\left\|#1\right\|}              
\newcommand{\triplet}[3]{\left( #1, #2, #3 \right) }             
\newcommand{\ProbSpace}{\triplet{\Omega}{\mathcal{F}}{\Prob}}    
\newcommand{\abs}[1]{\left| #1 \right|}                          

\newcommand{\quadraVari}[1]{\left\langle  #1  \right\rangle } 

\newcommand{\fle}{\rightarrow}

\newcommand{\operQuadraVari}[1]{\left\langle \!\left\langle #1  \right\rangle \!\right\rangle} 

\newcommand\restr[2]{{
  \left.\kern-\nulldelimiterspace 
  #1 
  \vphantom{\big|} 
  \right|_{#2} 
  }}

\title[Markov property SPDEs]{Markov property and path regularity for the solutions to SPDEs driven by cylindrical-martingale valued measures}

\author{S. Cambronero\,\orcidlink{0000-0001-6758-4942}$^1$}
\author{D. Campos\,\orcidlink{0000-0002-3608-1151}$^2$}
\author{C. A. Fonseca-Mora\,\orcidlink{0000-0002-9280-8212}$^3$}
\author{D. Mena\,\orcidlink{0000-0002-9443-391X}$^4$}
\address{Centro de Investigaci\'{o}n en Matem\'{a}tica Pura y Aplicada \\ Escuela de Matem\'{a}tica, Universidad de Costa Rica}
\email{$^1$ santiago.cambronero@ucr.ac.cr} 
\email{$^2$ josedavid.campos@ucr.ac.cr}
\email{$^3$ christianandres.fonseca@ucr.ac.cr}
\email{$^4$ dario.menaarias@ucr.ac.cr}

\begin{document}
\emergencystretch 3em


\begin{abstract}
In this paper we prove the Markov property for the solution to stochastic partial differential equations driven by a cylindrical orthogonal martingale-valued measure. We assume our coefficients are time-dependent and satisfy some growth and Lipschitz conditions. We also prove that for time-independent coefficients and under mild assumptions on the cylindrical orthogonal martingale-valued measure, the solutions to our stochastic partial differential equations are Feller.  Finally, in the case that the $C_{0}$-semigroup is quasi-contraction, we show that the solution to our stochastic partial differential equation possesses a c\`adl\`ag version. 
\end{abstract}

\maketitle

\section{Introduction}

Stochastic partial differential equations in infinite-dimensional spaces have become a fundamental framework for modeling systems affected by both spatial structure and random fluctuations. Such equations arise naturally in the modelling of complex systems, such as fluid dynamics, finance, population dynamics, quantum fields, and spatially extended phenomena. 

Let $H$ and $G$ be separable Hilbert spaces and $\mathcal{A}$ an algebra of Borel subsets of a topological space $U$.  Let $M=(M(t,A):t \geq 0, A \in \mathcal{A})$ be a cylindrical orthogonal martingale-valued measure on $H$, according to Definition $4.1$ in \cite{CCFM:SPDE}. These objects includes a large collection of other cylindrical and classical stochastic processes widely used in the literature as a noise, for example, the cylindrical square integrable martingales, Hilbert space-valued martingale-valued measures, square integrable L\'evy processes, among others (see Chapter 5 in \cite{CCFM:SPDE}).  Let $A$ be the infinitesimal generator to a $C_{0}$-semigroup of continuous linear operators on $G$. In this paper we study the Markov property and path regularity for the solutions to the following class of stochastic partial differential equations:
\begin{equation}\label{eqSPDEIntro}
dX_t = \left[ AX_t + B(t,X_t)\right]dt + \int_U F(t,u,X_t) M(dt,du),
\end{equation}
where the coefficients $B$ and $F$ satisfy appropriate measurability conditions (Section \ref{sectSolutionsSPDEs}). 

In \cite{CCFM:SPDE}, a theory was introduced for the existence and uniqueness of weak and mild solutions to \eqref{eqSPDEIntro} under some standard growth and Lipschitz type conditions on the coefficients (see Assumption \ref{assumpLipschitz}). 
In this paper, we are mainly interested in studying the Markov property and path regularity of those solutions. 

In Section \ref{sectCMVM} we start by recalling the basic definition and properties of the quadratic variation for a cylindrical orthogonal martingale-valued measure.  In Section \ref{sectStochIntegral} we review the rudiments of the construction of stochastic integrals defined with respect to these objects. In Section \ref{sectSolutionsSPDEs} we formulate the main hypothesis on the coefficients in \eqref{eqSPDEIntro}, including the growth and Lipschitz conditions which are sufficient to guarantee the existence of a unique solution. 

In Section \ref{sectMarkovProperty} we study the Markov property for the solutions. As some works suggest (see e.g. \cite{DaPratoZabczyk:StochasticEquations, GawareckiMandrekar, KumarRiedle, Ondrejat:2005, PeszatZabczykSPDE}), one might need to assume that the noise $M$ satisfies some independent increments property. Due to the cylindrical nature of $M(\cdot,A)$, we require it to have weakly independent increments  for any given $A \in \mathcal{A}$ (see Definition \ref{defiIndepenIncrements}). Our main result is Theorem \ref{theoremMarkov} were we prove the Markov property of the solutions to \eqref{eqSPDEIntro}. To illustrate the usefulness of the above result, in Example \ref{examCylinPoissonIntegrals} we introduce a cylindrical orthogonal martingale measure which is defined as the orthogonal sum of a sequence of compensated Poisson integrals (see \eqref{eqOrthogonalCMVM});  we prove that such an $M$ satisfies the weakly independent increments property and we show that the solution to a stochastic heat equation driven by this noise satisfies the Markov property. 

In Section \ref{sectFellerSolutions} we assume that our coefficients $B$ and $F$ are time-independent and we show that the solution to \eqref{eqSPDEIntro} is a Feller process. To do this, one requires some  form of time-invariance on the distributions of $M$, which again is formulated in a weak sense (see Definition \ref{defiTimeTranslationInvariance}). Our main results in this section are Theorem \ref{theoHomogeMarkovProcess}, where we show that the solution is an homogeneous Markov process, and Theorem \ref{theoTransitionSemigroIsFeller} where we prove that the transition semigroup is Feller. We end this section with Example \ref{examCMVMCylindriLevyProcesses}, where we prove how our theory can be applied to show the Markov property for the solutions to a stochastic partial differential equation driven by a cylindrical square integrable L\'evy process. The Markov property for solutions to SPDEs with general cylindrical L\'evy noise for linear additive equations was proven in \cite{KumarRiedle}. Even though we are assuming second moments for our cylindrical L\'evy noise it is worth mentioning that we are dealing with non-linear equations with multiplicative noise and time-dependent coefficients (see \eqref{eqSPDECylLevy}). 

It is important to clarify that our assumptions do not imply that $M(\cdot,A)$ is a cylindrical L\'evy process in the sense of \cite{ApplebaumRiedle:2010}. Indeed, for the Markov property in Section \ref{sectMarkovProperty} we only need that $M(\cdot,A)(h)$ is a real-valued square integral additive process, and for the Feller property in Section \ref{sectFellerSolutions} we only need that $M(\cdot,A)(h)$ is a real-valued square integrable L\'evy process (no need for finite-dimensional L\'evy property of any finite dimension as in \cite{ApplebaumRiedle:2010}). We expect that our general formulation 
for $M$ can be used to model new interesting examples of noises. 
In particular, the cylindrical orthogonal martingale-valued measure of Example \ref{examCylinPoissonIntegrals} satisfies the weakly independent increments property but does not usually satisfy the time-invariance on the distributions (see Remark \ref{remaCylinPoissonNoTimeInvariant}).  

Path properties of the solutions to \eqref{eqSPDEIntro} are discussed in Section \ref{sectTimeRegularity}. 
We assume that the $C_0$-semigroup $(S(t): t \geq 0)$ is a quasi-contraction semigroup, that is, there exists $\theta \geq 0$ such that $\norm{S(t)}_{\mathcal{L}(G)}\leq e^{\theta t}$ for all $t \geq 0$.  We prove a Kotelenez type inequality for the stochastic convolution with respect to a cylindrical orthogonal martingale-valued measure (see Theorem \ref{theoKotelenez}). Then, we prove the existence of a c\`adl\`ag version for the stochastic convolution (see Theorem \ref{theoStochConvoluCadlag}). Finally, in Theorem \ref{theoExistenceCadlagVersion} we prove the existence of a  c\`adl\`ag version for the solution to  \eqref{eqSPDEIntro}. Our result generalizes, in some directions, some of the works reported in \cite{ DaPratoKwapienZabczyk:1988,Kotelenez:1982, LiuZhai, LiuZhainote, PeszatZabczykregularity}, so we consider it a significant enrichment to the existing literature of time regularity of solutions to SPDEs in Hilbert spaces. In particular, most of these works consider only additive noise and time-independent coefficients (as for example generalized Ornstein-Uhlenbeck processes).  


\section{Preliminaries}\label{sectPreliminaries}

\subsection{Cylindrical martingale-valued measures and quadratic vartiation} \label{sectCMVM}

Throughout this work we fix a Hausdorff topological space $U$, which is Lusin in the sense that it is homeomorphic to a Borel subset of the line. We consider a ring $\mathcal{A}$ of Borel subsets of $U$ and  a cylindrical orthogonal martingale-valued measure $M=(M(t,A): t \geq 0, A \in \mathcal{A})$ on $H$, that is,  a collection  $(M(t,A): t \geq 0, A \in \mathcal{A})$ of cylindrical random variables on $H$ such that
\begin{enumerate} 
\item For each $A \in \mathcal{A}$, $M(0,A)(h)= 0$ $\Prob$-a.e. for all $h \in H$.
\item For each $A \in \mathcal{A}$, $M(A) = (M(t,A): t \geq 0)$, is a cylindrical mean-zero square integrable martingale, and for each $t > 0$ and $A \in \calA$, the map  
$M(t,A): H \rightarrow L^{0} \ProbSpace$ 
is continuous. 
\item If $t>0$ and $h \in H$, $M(t,\cdot)(h): \mathcal{A} \rightarrow L^{2} \ProbSpace$ is a $\sigma$-finite $L^{2}$-valued measure.
\item Given $t>0$ and $h \in H$, $\quadraVari{M(A)(h),M(B)(h)}_{t}=0$ whenever $A,B\in \mathcal{A}$ are disjoint. 
\end{enumerate}
Associated to $M$ there exists a collection random predictable $\sigma$-finite measures $(\nu_{h}: h \in H)$, called the \emph{intensity measures} of $M$, with 
the property that for every $t \geq 0$  and $A \in \mathcal{A}$, we have $ \nu_{h} (\omega) ([0, t] \times A) = \quadraVari{M(A)(h)}_{t}(\omega)$ $\Prob$-a.e. 

We will assume that $M$ has a unique (predictable) quadratic variation (see Theorem 5.10 in \cite{CCFM:SPDE} on sufficient conditions for its existence), that is, a random measure $\operQuadraVari{M}: \Omega \rightarrow \mathcal{M}_{+}(\R_{+} \times U, \mathcal{B}(\R_{+}) \otimes \mathcal{B}(U) )$ that satisfies: 
\begin{enumerate}    
    \item Given $t \geq 0$ and $A \in \mathcal{A}$, for $\Prob$-a.e. $\omega \in \Omega$ we have $\operQuadraVari{M}(\omega)([0,t] \times A) < \infty$.
    \item \label{propertyUpperBoundNuX} $\operQuadraVari{M}$ is a minimal element  for the collection of all random measures $\zeta: \Omega \rightarrow \mathcal{M}_{+}(\R_{+} \times U, \mathcal{B}(\R_{+}) \otimes \mathcal{B}(U) )$ with the property:   $\forall\, h \in H$ with $\| h \| = 1$,  $\nu_{h} \leq \zeta$. 
\end{enumerate} 
We denote property \ref{propertyUpperBoundNuX} as $\displaystyle \operQuadraVari{M}=\sup_{\norm{h}=1} \nu_{h}$. 

We will further assume that the quadratic variation of $M$ satisfies: for $\Prob$-a.e. $\omega \in \Omega$,
\begin{equation}\label{eqBoundedrangequadraticvariation}
    \sup_{A \in \calA} \operQuadraVari{M}(\omega)([0,T]\times A) <\infty.
\end{equation}
Then by  Theorem 5.21 in \cite{CCFM:SPDE}, given $T>0$,  there exists a  process $Q_{M}: \Omega \times [0,T]  \times U \rightarrow \mathcal{L}(H,H)$ such that for all $h_{1},h_{2} \in H$, $0 \leq t\leq T$ and $A \in \mathcal{A}$, $\Prob$-a.e. $\omega \in \Omega$,
\begin{equation}\label{existenceofQ}
\quadraVari{M(A)(h_{1}),M(A)(h_{2})}_{t}(\omega) = \int_{[0,t]\times A} \inner{Q_{M}(\omega,r,u)h_{1}}{h_{2}}_H \, \operQuadraVari{M} (\omega)(dr,du)    
\end{equation}
for all $h_{1}, h_{2} \in H$, $C \in \mathcal{B}([0,T]) \otimes \mathcal{B}(U)$. Moreover, the following properties hold:
\begin{enumerate}
    \item For every $h_{1}, h_{2} \in H$, the mapping $(\omega,r,u) \mapsto \inner{Q_{M}(\omega,r,u)h_{1}}{h_{2}}$ is predictable, that is, $\calP_T \otimes \calB(U)$-measurable. \label{Qpredictable}
    \item $\Prob$-a.e. $\omega \in \Omega$, $Q_{M}(\omega,\cdot,\cdot)$ is positive and symmetric $\operQuadraVari{M}$-a.e. \label{Qpositiveandsymmetric}
    \item $\Prob$-a.e. $\omega \in \Omega$, $\norm{Q_{M}(\omega,\cdot,\cdot)}_{\mathcal{L}(H,H)}=1$,  $\operQuadraVari{M}$-a.e. \label{normoneoftheoperatorQ}
\end{enumerate}

\subsection{Stochastic integration}\label{sectStochIntegral}

The basic class of integrands will be denoted by $\Lambda^2(M,T)$ and corresponds to the collection of all families of operators $\Phi(\omega,t,u),$ indexed by $(\omega,t,u)\in \Omega\times [0,T]\times U$, such that:
\begin{enumerate}
\item \label{domainIntegrands} For each $(\omega,t,u) \in \Omega\times [0,T]\times U$, $\Phi(\omega,t,u)\in \mathcal{HS}(H,G)$. 
\item \label{predictabilityIntegrands} For every $h \in H$, the mapping $(\omega,t,u) \mapsto \Phi(\omega,t,u) h$, is $\mathcal{P}_{T}\otimes \mathcal{B}(U)/\mathcal{B}(G)$-measurable.  
\item \label{integrabilityIntegrands} $\norm{\Phi}^2_{\Lambda^2(M,T)}$ is finite, where this quantity is defined by
\begin{equation} \label{NewIntegrands}
 \norm{\Phi}^2_{\Lambda^2(M,T)} =
\Exp \int_{[0,T]\times U} \| \Phi(\omega,t,u)\circ Q_{M}^{1/2}(\omega,t,u)\|^2_{\mathcal{HS}(H,G)} \, \operQuadraVari{M}(dt,du).    
\end{equation}
\end{enumerate}

The space $\Lambda^2(M,T)$ is complete with the Hilbertian seminorm given by (\ref{NewIntegrands}). In order to define a stochastic integral for this space, we first consider a class of simple operator-valued maps, denoted by $\mathcal{S}(M,T)$, which is contained in $\Lambda^2(M,T)$. Elements in $\mathcal{S}=\mathcal{S}(M,T)$ can be expressed in the form
\begin{equation}
\label{eqSimpleIntegrands}
\Phi(\omega,r,u)=  \sum_{i=1}^{n}\sum_{j=1}^{m} \ind{(s_{i}, t_{i}]}{r} \ind{F_{i}}{\omega} \ind{A_{j}}{u} S_{i,j},
\end{equation}
for all $r \in [0,T]$, $\omega \in \Omega$, $u \in U$, where $m$, $n \in \N$, and for $i=1, \dots, n$, $j=1, \dots, m$, $0\leq s_{i}<t_{i} \leq T$, $F_{i} \in \mathcal{F}_{s_{i}}$, $A_{j} \in \mathcal{R}$ and $S_{i,j} \in \mathcal{HS}(H,G)$. For $\Phi\in \mathcal{S}(M,T)$ we define its stochastic integral $I(\Phi)$ as the $G$-valued process given by 
\begin{equation}
\label{NewDefIntSimpleIntegrand}
I_t(\Phi) : = \int_0^t\! \! \int_U \Phi(s,u) M(ds,du) = 
\sum_{i=1}^{n} \sum_{j=1}^{m} \caract_{F_i}(\omega) Y_{i,j}(t)
\end{equation}
where $Y_{i,j}$ is the square integrable martingale taking values in $G$ and satisfying (see \cite{AlvaradoFonseca:2021})
$$
\left( Y_{i,j}(t),g \right)_{G} = M\Bigl((s_i \land t, t_i\land t], A_j\Bigr) (S_{ij}^*g).
$$
The map $I:\mathcal{S}(M,T) \to \mathcal{M}^2(G)$ is linear and continuous and satisfies the It\^{o} isometry, given by
\begin{equation}
\label{eqItoIsometrySimpleIntegrands}
\Exp \left[ \norm{I_{t}(\Phi)}^{2} \right] = \Exp\int_{0}^{t} \int_{U} \norm{ \Phi(\omega, s,u)\circ Q^{1/2}_M(\omega,s,u)}_{\mathcal{HS}(H,G)}^{2}\operQuadraVari{M}(ds,du).
\end{equation}
Since $\mathcal{S}(M,T)$ is dense in $\Lambda^2(M,T)$, the map $I$ has a continuous linear extension $$I: \Lambda^2(M,T) \fle \mathcal{M}_T^2(G),$$ where (\ref{eqItoIsometrySimpleIntegrands}) holds for $\Phi \in \Lambda^2(M,T)$. In addition, $I(\Phi)$ is a continuous process provided that, for each $A \in \mathcal{A}$ and $h \in H$, the real-valued stochastic process $M(\cdot, A)(h)$ is continuous (see Proposition 6.14 in \cite{CCFM:SPDE}).

By following a standard localization procedure, the stochastic integral can be further defined for locally square integrable integrands satisfying
\begin{equation}
\label{LocallyIntegrable}
\Prob \left(
\int_{[0,T]\times U} \| \Phi(\omega,s,u)\circ Q_M^{1/2}(\omega,s,u)\|^2_{\mathcal{HS}(H,G)} \, \operQuadraVari{M}(ds,du) < \infty \right) = 1.
\end{equation}
The reader is referred to Section 6.3 in \cite{CCFM:SPDE} for further details.

For our study of the Markov properties of SPDEs we will need the following definition of generated $\sigma$-algebra by $M$. 

\begin{definition}
For $0 \leq s \leq t \leq T$, let $\mathcal{F}^{M}_{s,t}=\sigma \left( M((u,t],A)(g): s \leq u \leq t, A \in \calA, g \in G \right) $. 
\end{definition}

\begin{lemma}\label{lemmaMeasurabPropStochIntegral}
Assume $\Phi \in \Lambda^{2}(M,T)$. For $0 \leq s \leq t \leq T$, the stochastic integral 
$$\int_{s}^{t}\!\! \int_{U} \, \Phi(r,u) \, M(dr,du)
$$ 
is $\mathcal{F}^{M}_{s,t}$-measurable. 
\end{lemma}
\begin{proof}
By Proposition 6.13 in \cite{CCFM:SPDE} we have
$$ \int_{s}^{t}\!\! \int_{U} \, \Phi(r,u) \, M(dr,du) = \int_{0}^{t}\!\! \int_{U} \mathbbm{1}_{(s,t]} \, \Phi(r,u) \, M(dr,du),$$
so we must show that the stochastic integral of $\mathbbm{1}_{(s,t]} \, \Phi$ is $\mathcal{F}^{M}_{s,t}$-measurable. 

Assume first that $\Phi$ is of the elementary form $\Phi(\omega,s,u)= \mathbbm{1}_{(s_{0},t_{0}]}(s) \mathbbm{1}_{F}(\omega) \mathbbm{1}_{A}(u) S$, 
where $0 \leq s_{0} < t_{0} \leq T$, $F \in \mathcal{F}_{s_{0}}$, $A \in \mathcal{A}$ and $S\in \mathcal{HS}(H,G)$. By the definition of the stochastic integral, for every $g \in G$ we have
\begin{flalign*}
\left(  \int_{0}^{t}\!\! \int_{U} \mathbbm{1}_{(s,t]} \, \Phi(r,u) \, M(dr,du), g \right)_{G} 
= 
\begin{cases}
 0 & \mbox{ if } t_{0} < s \mbox{ or } s_{0}>t, \\
 \caract_{F} M((s_{0} \vee s, t_{0} \wedge t], A ) (S^{*} g) & \mbox{ if } t_{0} \geq s \mbox{ and } s_{0} \leq t,
\end{cases}
\end{flalign*}
which is $\mathcal{F}^{M}_{s,t}$-measurable. Since $G$ is a separable Hilbert space, we conclude that the stochastic integral $\displaystyle \int_{s}^{t}\!\! \int_{U} \, \Phi(r,u) \, M(dr,du)$ is $\mathcal{F}^{M}_{s,t}$-measurable. By linearity, this measurability extends to every simple integrand. For general $\Phi \in \Lambda^{2}(M,T)$, we know that $\Phi$ is the limit in $\Lambda^{2}(M,T)$ of simple integrands $\Phi_{n}$, and by the It\^{o} isometry
$$ 
\int_{0}^{t}\!\! \int_{U} \mathbbm{1}_{(s,t]} \, \Phi(r,u) \, M(dr,du) = L^2\text{-}\lim \int_{0}^{t}\!\! \int_{U} \mathbbm{1}_{(s,t]} \Phi_{n}(r,u)  \, M(dr,du).
$$
Being each integral on the right $\mathcal{F}_{s,t}^M$-measurable, the result follows.
\end{proof}

\subsection{Solutions to SPDEs}\label{sectSolutionsSPDEs}

Consider the following class of stochastic evolution equations:
\begin{equation}\label{EqSPDE}
  dX_t = \left[ AX_t + B(t,X_t)\right]dt + \int_U F(t,u,X_t) M(dt,du),
\end{equation}
where we will assume the following: 
\begin{enumerate}
    \item $A$ is the infinitesimal  generator for a $G$-valued $C_{0}$-semigroup $(S(t): t \geq 0)$. 
    \item $B: \R_+ \times G \to G$ is $\calB (\R_{+}) \otimes \calB(G)/\calB(G)$-measurable. 
    \item $F$ is a family of operators  $(F(t, u,g): t\geq 0, u \in U, g \in G)$ such that  
    \begin{enumerate}
        \item For all $ t\geq 0, u \in U, g \in G$, $F(t, u,g) \in \mathcal{HS}(H, G)$.
        \item For every $h \in H$, the mapping $(\omega, t,u,g) \mapsto F(t,u,g)\circ Q_M^{1/2}(\omega, t,u)(h)$ is $\calP_{T} \otimes   \calB(U)  \otimes \calB(G)/\calB(G)$-measurable. 
    \end{enumerate}
\end{enumerate}

We will be interested in weak and mild solutions for \eqref{EqSPDE}. Recall that a $G$-valued predictable process $X=(X_{t}: t \geq 0)$ is a \emph{weak solution} for equation \eqref{EqSPDE} if for all $g\in \mbox{Dom}(A^*)$  and $t \geq 0$, $\Prob$-a.e.
\begin{multline}\label{eqDefiWeakSolution}
    \left(X_t,g\right)_G  = \left(X_0,g\right)_G + \int_0^t \left[ \left(X_s,A^* g\right)_G  
    + \left(B(s,X_s),g\right)_G \right] ds \\
    + \int_0^t\! \! \int_U \left(F(s,u,X_s),g\right)_G M(ds,du),
\end{multline}
where the first integral on the right-hand side is defined $\Prob$-a.e. as a Lebesgue integral, and the second integral is the (real-valued) stochastic integral of the family $\{ \left(F(s,u,X_{s}(\omega),g\right)_G:  \omega \in \Omega, s \geq 0, u \in U\}$.

On the other hand, a $G$-valued predictable process $X=(X_{t}: t \geq 0)$ is a  \emph{mild solution} for equation \eqref{EqSPDE}  if for all $t  \geq 0$, $\Prob$-a.e.
\begin{equation}\label{EqStochIntDiffEq}
  X_t = S(t)X_0 + \int_0^t S(t-s)B(s,X_s)ds + \int_0^t\! \! \int_U S(t-s) F(s,u,X_s)M(ds,du).  
\end{equation}
Here the first integral is defined $\Prob$-a.e. as a Bochner integral, and the second integral is the stochastic integral of the family $\{ \mathbbm{1}_{[0,t]}(s) S(t-s)F(t,u,X_{s}(\omega)):  \omega \in \Omega, s \in [0,t], u \in U\}$.

\begin{assumption}
\label{assumpLipschitz} Assume that the coefficients $B$ and $F$ satisfy the following:
\begin{enumerate}
\item There is a constant $C_B > 0$ such that for all $t \in [0,T]$ and $g, h \in G$ we have
\begin{align}
\label{eqBoundBB}
    \| B(t, g)  \|_G & \leq C_B(1 + \|  g \|_G) \\
    \label{eqBoundBL}
    \| B(t, g) - B(t ,h)  \|_G & \leq C_B \|  g  - h \|_G
\end{align}

\item There is a constant $C_F > 0$ such that for all $t \in [0,T]$, $\omega \in \Omega$ and $g, h:[0,T] \to G$ c\`adl\`ag we have
\begin{align}
\label{eqBoundIntFQ}
    & \int_0^t\! \! \int_U \| F(s, u, g(s)) \circ Q^{1/2}_{M} \|_{\mathcal{HS}(H,G)}^2 \operQuadraVari{M}(ds, du)  \leq C_F\left(1 + \int_0^t \|  g(s) \|^2 \, ds \right) \\
    \begin{split}
\label{eqLipIntFQ}
        \int_0^t\! \! \int_U \| ( F(s, u, g(s)) - F(s, u, h(s)) )  \circ Q^{1/2}_{M} \|_{\mathcal{HS}(H,G)}^2 \operQuadraVari{M}(ds, du) \\ \leq C_F \int_0^t \|  g(s) -h(s)\|^2 \, ds
    \end{split}
\end{align}
\end{enumerate}
\end{assumption}

It is shown in Theorem 7.8 in \cite{CCFM:SPDE} that for a $\mathcal{F}_0$ measurable square integrable random variable $X_0$, the stochastic evolution equation \eqref{EqSPDE} has a unique weak solution $X=(X_{t}: t \geq 0)$ (which is its mild solution) with initial value $X_0$. Moreover, for every $T>0 $ we have $\displaystyle \Exp \int_{0}^{T} \norm{X_{t}}^{2}_{G} dt <\infty.$  Furthermore, $X=(X_{t}: t \geq 0)$  is the unique weak solution to \eqref{EqSPDE}.

\section{The Markov property for the solutions}\label{sectMarkovProperty}


In this section we prove that under additional assumptions on $M$ the unique solution to \eqref{EqSPDE} is  a Markov process with respect to the filtration $\left(\mathcal{F}^{M,\xi}_{t}: t \geq 0 \right)$, where $\mathcal{F}^{M,\xi}_{t}=\sigma \left( \mathcal{F}^{M}_{t_{0},t} \cup \sigma (\xi) \right)$. In the following lemma we list some properties of the solutions that will be of reference in our proofs:

\begin{lemma}
\label{lemmaExistenUniqueSolutions}
For every $t_{0} \geq 0$ and every $G$-valued $\mathcal{F}_{t_{0}}$-measurable integrable random variable $\xi$, the stochastic evolution equation \eqref{EqSPDE}
has a unique $\left(\mathcal{F}^{M,\xi}_{t}: t \geq t_{0} \right)$-adapted weak solution $(X_{t}: t \geq t_{0})$ with initial value $\xi$ satisfying $\displaystyle \Exp \int_{t_{0}}^{T} \norm{X_{t}}^{2}_{G} dt <\infty$ for every $t>t_{0} $ . This solution is given by:
\begin{equation}\label{eqMildSoluAtT0}
  X_t = S(t-t_{0})\xi + \int_{t_{0}}^t \, S(t-s)B(s,X_s) \, ds + \int_{t_{0}}^t \!\!\int_U \,  S(t-s) F(s,u,X_s) \, M(ds,du). \end{equation}
\end{lemma}
\begin{proof}
Let $t_{0} \geq 0$. By replicating the arguments in the proof of Theorem 7.8 in \cite{CCFM:SPDE}, for every $G$-valued $\mathcal{F}_{t_{0}}$-measurable square integrable random variable $\xi$, the stochastic evolution equation \eqref{EqSPDE}
has a unique weak solution $(X_{t}: t \geq t_{0})$ with initial value $\xi$, given by \eqref{eqMildSoluAtT0}; it can be approximated on the space $L^{2}(\Omega \times [t_{0},T], \mathcal{P}_{T}, \Prob \otimes dt; G)$ by the sequence 
\begin{flalign*}
& X_{t}^{0}=S(t-t_{0})\xi, \\
& X^{n+1}_{t}=S(t-t_{0})\xi + \int_{t_{0}}^t \, S(t-r)B(r,X^{n}_r) \, dr + \int_{t_{0}}^t\!\! \int_U \,  S(t-r) F(r,u,X^{n}_r) \, M(dr,du). 
\end{flalign*}
Each $X^{n}_{t}$ is  $\mathcal{F}^{M,\xi}_{t}$-measurable by Lemma \ref{lemmaMeasurabPropStochIntegral},  
so $X_t$ is $\mathcal{F}^{M,\xi}_{t}$-measurable. 
\end{proof}

\begin{notation} When we need to emphasize the dependence of the solution on $t_{0}$ and $\xi$, we denote it by $X(t,t_{0};\xi)$.
\end{notation}

For further developments we will need the following result on the continuity of the solutions with respect to the initial data. 

\begin{lemma}\label{lemmaContiDepenInitialData}
Let $0 \leq t_{0} \leq T$. For any two $G$-valued $\mathcal{F}_{t_{0}}$-measurable square integrable random variables $\xi$ and $\eta$, there exists $C_{T}>0$ such that 
\begin{equation*}
 \Exp \norm{ X(t,t_{0};\xi) -X(t,t_{0};\eta) }^{2}_{G}   \leq 
 C_{T} \Exp \norm{\xi-\eta}^{2}_{G}. 
\end{equation*}
\end{lemma}
\begin{proof}
Following the estimates in the proof of Theorem 7.8 in \cite{CCFM:SPDE}, one can show the existence of some constants $C_{1,T}, C_{2,T}>0$ such that, for every $t_{0} \leq t \leq T$,
\begin{multline}
 \Exp \norm{ X(t,t_{0};\xi)-X(t,t_{0};\eta)}^{2}_{G}  \\
 \leq C_{1,T} \Exp \norm{\xi-\eta}^{2}_{G}+ C_{2,T} \Exp \int_{t_{0}}^{t} \norm{ X(s,t_{0};\xi)-X(s,t_{0};\eta)}^{2}_{G} ds.  
\end{multline}
The result is obtained as a consequence of Gronwall's lemma. 
\end{proof}

We will need the following mean square continuity  of the solution. 

\begin{lemma}\label{lemmaMeanSquareContinuity}
For every $t_{0} \geq 0$ and every $G$-valued $\mathcal{F}_{t_{0}}$-measurable integrable random variable $\xi$, we have for $t \geq t_{0}$, 
$$ \lim_{s \rightarrow t} \Exp \norm{ X(s,t_{0};g)-X(t,t_{0};g)}^{2}_{G}=0$$
\end{lemma}
\begin{proof}
    Let $0 \leq s \leq t$ and $g \in G$. We start with the identity 
\begin{eqnarray*}
X(t,0;g)-X(s,0;g) 
& = & X(t,s;X(s,0;g))-X(s,0;g) \\
& = & (S(t-s)-I) X(t,s;g) + \int_{s}^{t} S(s-r) B(X(r,0;g)) dr \\
& {} & + \int_{s}^{t} \int_{U} S(s-r)F(u,X(r,0;g)) M(dr,du).  
\end{eqnarray*}
Then, from the estimates in the proof of Theorem 7.8 in \cite{CCFM:SPDE} we arrive at the inequality
\begin{flalign}
& \Exp \left[ \norm{X(t,0;g)-X(s,0;g)}_{G}^{2} \right] \nonumber \\ 
& \leq  3 \Exp \left[ \norm{ (S(t-s)-I) X(t,s;g) }_{G}^{2} \right]  + 3 C_{T} \int_{s}^{t} \Exp \left[ \norm{X(r,0;g)}_{G}^{2} \right] dr, \label{eqContiSolutionInTimeVariable}
\end{flalign}
which shows that the limit from the left is zero. An analogous argument can be used for the limit from the right.
\end{proof}

\begin{definition}\label{defiIndepenIncrements}
We say that $M$ has \emph{independent increments} if for every $0 \leq s < t $, $h \in H$ and $A \in \mathcal{A}$, $M((s,t],A)(h)$ is independent of $\mathcal{F}_{s}$. 
\end{definition}

Denote by $ B_{b}(G)$ the set of bounded Borel-measurable real-valued functions on $G$. 

\begin{definition}
For $g \in G$, $\varphi \in B_{b}(G)$, $0 \leq s \leq t$, define
$$ \left( P_{s,t} \varphi \right)(g) = \Exp \left[ \varphi (X(t,s;g)) \right],$$
and for all $0 \leq s \leq t$, $g \in G$,  and $\Gamma \in \mathcal{B}(G)$,  
$$ P(s,g,t,\Gamma)\defeq (P_{s,t} \mathbbm{1}_{\Gamma})(g) = \Prob \left( X(t,s,g) \in \Gamma \right). $$
\end{definition}

The following is the main result of this section. 

\begin{theorem}
\label{theoremMarkov}
Assume $M$ has independent increments. For $u \geq 0$, let $\xi$ be a   $G$-valued $\mathcal{F}_{u}$-measurable square integrable random variable. 
Then, for every $\varphi \in B_{b}(G)$, $u \leq s \leq t$, we have
\begin{equation}\label{eqMarkovProperty}
\Exp \left[ \varphi(X(t,u;\xi)) \, \vline \, \mathcal{F}^{M,\xi}_{s} \right] =  \left( P_{s,t} \varphi \right)(X(s,u;\xi)).    
\end{equation}
In particular, $(X(t,u; \xi): t \geq u)$ is a Markov process.    
\end{theorem}
\begin{proof} 
By uniqueness of solutions (Lemma \ref{lemmaExistenUniqueSolutions}), for $u \leq s \leq t$, we have $\Prob$-a.e. 
$$ X(t,u; \xi)=X(t,s;X(s,u; \xi)).$$
Hence, to show \eqref{eqMarkovProperty} it is equivalent to prove that 
$$ \Exp \left[ \varphi(X(t,s;X(s,u; \xi))) \, \vline \, \mathcal{F}^{M,\xi}_{s} \right] =  \left( P_{s,t} \varphi \right)(X(s,u;\xi)).    $$
Thus, it suffices to show the more general property:
\begin{equation}\label{eqStrongerMarkovPropertyInProof}
\Exp \left[ \varphi(X(t,s;\eta)) \, \vline \, \mathcal{F}^{M,\xi}_{s} \right] =  \left( P_{s,t} \varphi \right)(\eta),    
\end{equation}
for a $G$-valued $\sigma(X(s,u;\xi))$-measurable square integrable random variable $\eta$. Furthermore, by the functional form of the monotone class theorem, it is enough to show that  \eqref{eqStrongerMarkovPropertyInProof} holds for $\varphi \in C_{b}(G)$. 

Now, observe that by Lemma \ref{lemmaMeasurabPropStochIntegral}, for every $g \in G$ the solution $X(t,s;g)$ is $\mathcal{F}^{M}_{s,t}$-measurable. Hence, by the independent increments property of $M$ we have that $X(t,s;g)$ is independent of $\mathcal{F}^{M,\xi}_{s}$. Therefore, 
$$ \Exp \left[ \varphi(X(t,s;g)) \, \vline \, \mathcal{F}^{M,\xi}_{s} \right] =  \left( P_{s,t} \varphi \right)(g).  $$
Now, consider a simple function 
\begin{equation}\label{eqSimpleFunctionEtaProofMarkov}
\eta= \sum_{j=1}^{n} g_{j} \mathbbm{1}_{\Gamma_{j}}(X(s,u;\xi)),     
\end{equation}
where $\{ \Gamma_{j}: j=1,\dots, n \}$ is a measurable partition of $G$, and $g_{1}, \dots, g_{n} \in G$. By uniqueness of solutions 
we can show that
$$ X(t,s;\eta)=\sum_{j=1}^{n} X(t,s;g_{j}) \mathbbm{1}_{\Gamma_{j}}(X(s,u;\xi)).$$
Then, by the independence of $X(t,s;g_{j})$ with $\mathcal{F}^{M,\xi}_{s}$, we have 
\begin{eqnarray*}
\Exp \left[ \varphi(X(t,s;\eta)) \, \vline \, \mathcal{F}^{M,\xi}_{s} \right] 
& = & \sum_{j=1}^{n} \Exp \left[ \varphi(X(t,u;g_{j})) \mathbbm{1}_{\Gamma_{j}}(X(s,u;\xi)) \, \vline \, \mathcal{F}^{M,\xi}_{s} \right] \\
& = & \sum_{j=1}^{n} \left( P_{s,t} \varphi \right)(g_{j}) \mathbbm{1}_{\Gamma_{j}}(X(s,u;\xi))  =  \left( P_{s,t} \varphi \right)(\eta).    
\end{eqnarray*}
Assume that $\eta$ satisfies $\Exp \left[ \norm{\eta}^{2}_{G} \right] < \infty$. Then there exists a sequence of simple functions $(\eta_{n}: n \in \N)$ in the form of \eqref{eqSimpleFunctionEtaProofMarkov} such that $\Exp \left[ \norm{\eta_{n}}^{2}_{G} \right] < \infty$ and $\Exp \left[ \norm{\eta_{n}-\eta}^{2}_{G} \right] \rightarrow 0$. By Lemma \ref{lemmaContiDepenInitialData} we have 
$$ \Exp \left[ \norm{X(t,s;\eta_{n})-X(t,s;\eta)}^{2}_{G} \right] \rightarrow 0. $$
We can select a subsequence $(n_{k}: k \in \N)$ such that $X(t,s;\eta_{n_{k}})\rightarrow X(t,s;\eta)$ $\Prob$-a.e. Since $\varphi \in C_{b}(G)$, we have
\begin{eqnarray*}
\Exp \left[ \varphi(X(t,s;\eta)) \, \vline \, \mathcal{F}^{M,\xi}_{s} \right] 
& = & \lim_{k \rightarrow \infty} \Exp \left[ \varphi(X(t,s;\eta_{n_{k}})) \, \vline \, \mathcal{F}^{M,\xi}_{s} \right]  \\
& = &  \lim_{k \rightarrow \infty} \left( P_{s,t} \varphi \right)(\eta_{n_{k}}) =  \left( P_{s,t} \varphi \right)(\eta).    
\end{eqnarray*}
This proves the desired identity \eqref{eqStrongerMarkovPropertyInProof}. 
\end{proof}

\begin{proposition}\label{propChapmanKolmogorov}
Assume $M$ has independent increments. For arbitrary $\varphi \in B_{b}(G)$, $\Gamma \in \mathcal{B}(G)$, $0 \leq s \leq u \leq t$, and $g \in G$,
$$ P_{s,u}(P_{u,t} \varphi)(g)= (P_{s,t} \varphi)(g),$$
and
$$P(s,g,t,\Gamma)=\int_{G} P(u,y,t,\Gamma) P(s,g,u,dy) $$
\end{proposition}

\begin{proof}
It is an immediate consequence of Theorem \ref{theoremMarkov}.
\end{proof}

\begin{example}\label{examCylinPoissonIntegrals}
Let $(\widetilde N_k)_{k\ge1}$ be independent compensated Poisson random measures on $[0,\infty)\times U$, each with compensator $ds\,\nu(du)$. 
Let $a:[0,\infty)\times U \to \mathbb{R}$ be measurable and assume
$$
\int_0^T \!\! \int_U |a(s,u)|^2\,ds\,\nu(du) < \infty, 
\qquad \forall t>0.
$$
Let $(e_{k}: k \in \N)$ be an orthonormal basis in $H$. Define, for $h\in H$,
$$
M(t,A)(h)
:=
\sum_{k=1}^\infty \inner{h}{e_k} \int_0^t \int_A a(s,u)\,\widetilde N_k(ds,du).
$$
We will show that $M$ is a cylindrical orthogonal martingale-valued measure with weakly independent increments. To do this, set
$$
Y_k(t,A)
:=
\int_0^t \int_A a(s,u)\,\widetilde{N}_k(ds,du),
\qquad k\ge1.
$$
Observe that 
$$
\mathbb{E}|Y_k(t,A)|^2
=
\int_0^t \int_A |a(s,u)|^2\,ds\,\nu(du),
$$
hence each $(Y_{k}(t,A):t \geq 0)$ is a real-valued mean-zero square integrable martingale. Thus one can easily conclude that each of these is a real-valued cylindrical orthogonal martingale-valued measure. Moreover, since the $\widetilde{N}_k$ are independent, so the $Y_{k}$ are.  Since
\begin{equation}\label{eqOrthogonalCMVM}
 M(t,A)(h)=\sum_{k=1}^\infty \inner{h}{e_k} Y_k(t,A),    
\end{equation}
then one can easily conclude that 
$$ 
\mathbb{E}|M(t,A)(h)|^2
=
\sum_{k=1}^\infty |\inner{h}{e_k} |^2 \mathbb{E}|Y_k(t,A)|^2
=
\|h\|^2 \int_0^t \int_A |a(s,u)|^2\,ds\,\nu(du).
$$
Thus $M(t,A):H\to L^2(\Omega, \mathcal{F},\Prob)$ is well-defined, linear and continuous. Then it is not difficult to conclude that $M$ is a cylindrical orthogonal martingale-valued measure. We can check that $M$ has weakly independent increments. In effect, for $0\le s<t$, $A \in \mathcal{A}$, $h \in H$,
$$ 
M(t,A)(h)-M(s,A)(h)
=
\sum_{k=1}^\infty \inner{h}{e_k} 
\int_s^t \int_A a(r,u)\,\widetilde N_k(dr,du).
$$
Since each $\widetilde N_k$ has independent increments on disjoint time intervals
and the family $(\widetilde N_k)$ is independent, it follows that for any finite collection
$h_1,\dots,h_m\in H$, the vector $(M(t,A)(h_{1}), \dots, M(t,A)(h_{n}))$,
has independent increments. Hence $M$ has weakly independent increments.
 
Let $\mathcal{O}$ be a non-empty open bounded subset of $\R^{d}$. Let $G=L^{2}(\mathcal{O})$.  Observe that the integrability condition for a stochastic integrand $\Phi$ is the following:
$$
\Exp \int_{0}^{T} \int_{U} |a(s,u)|^2 \norm{\Phi(r,u)}^{2}_{\mathcal{HS}(H,L^{2}(\mathcal{O}))} \,\nu(du)\,ds. < \infty.   $$

It is well-known that with Dirichlet boundary conditions, $A=\Delta$ is the generator for a strongly continuous $C_{0}$-semigroup $(S(t): t \geq 0)$ on $L^{2}(\mathcal{O})$ with domain $\mbox{Dom}(\Delta)=W^{1,2}_{0}(\mathcal{O})$. 

This is an equicontinuous semigroup,  known as the \emph{heat semigroup} and it is defined as: $S(0)=I$ and for each $t>0$, 
\begin{equation*}
(S(t)g) (x) \defeq \frac{1}{(4 \pi t)^{d/2}}  \int_{\mathcal{O}} e^{-\norm{x-y}^{2}/4t} g(y) dy, \quad \forall g \in  L^{2}(\mathcal{O}), \, x \in \mathcal{O}. 
\end{equation*}
Consider equation \eqref{EqSPDE} for $G=L^{2}(\mathcal{O})$, $B(t,g)=0$ and $F(t,u,g) \in \mathcal{HS}(H,L^{2}(\mathcal{O}))$. Then, by Theorem \ref{theoremMarkov} the mild solution 
\begin{align*}
 X_t & = S(t)X_0 + \int_0^t\! \! \int_{U} S(t-s) F(s,u,X_s) M(ds,du) \\
 & = S(t)X_0 + \sum_{k=1}^{\infty} \int_0^t\! \! \int_{U} a(s,u) S(t-s) F(s,u,X_s) e_{k} \,\widetilde N_k(ds,du),    
\end{align*}
 is a Markov process with values in $L^{2}(\mathcal{O})$.  More precisely, 
\begin{align*}
 X_t(x) & = \frac{1}{(4 \pi t)^{d/2}}  \int_{\mathcal{O}} e^{-\norm{x-y}^{2}/4t} X_0(y) dy \\
 & \quad + \frac{1}{(4 \pi t)^{d/2}} \sum_{k=1}^{\infty} \int_0^t\! \! \int_{U} \int_{\mathcal{O}} e^{-\norm{x-y}^{2}/4t} a(s,u) (F(s,u,X_s) e_{k}) (y)\, dy \widetilde{N}_{k}(ds,du) 
\end{align*}
\end{example}

\section{Feller solutions}\label{sectFellerSolutions}

\begin{definition}\label{defiTimeTranslationInvariance}
We say that $M$ has a \emph{time translation invariant distribution} if for every $t>0$, $0 \leq s \leq r $, $h \in H$ and $A \in \mathcal{A}$, $M((s+t,r+t],A)(h)$ and $M((s,r],A)(h)$ have the same distribution.  
\end{definition}

\begin{assumption}
\label{assumptionBF}
From now on, we assume that $B$ and $F$ do not depend on the time variable $t$ and that $M$ has independent increments and a time translation invariant distribution.
\end{assumption}

In this case, equation \eqref{EqSPDE} takes the form
\begin{equation}\label{EqSPDEFeller}
  dX_t = \left[ AX_t + B(X_t)\right]dt + \int_U F(u,X_t) M(dt,du).  
\end{equation}

\begin{theorem}\label{theoHomogeMarkovProcess}
For all $0 \leq s \leq t$, $P_{s,t}=P_{0,t-s}$. In particular, for every $0 \leq s \leq t$, $g \in G$, $\Gamma \in \mathcal{B}(G)$, 
$$
 P(s,g,t,\Gamma)=P(0,g, t-s,\Gamma).    
$$
\end{theorem}
\begin{proof}
Let $t>0$. Define a cylindrical martingale-valued measure $M^{t}$ by 
$M^{t}((s,r],A)\defeq M((s+t,r+t],A)$ $\forall 0 \leq s \leq r$, $A \in \mathcal{A}$. From the corresponding properties of $M$, it is clear that $M^{t}$ has a unique quadratic variation $\operQuadraVari{M^{t}}$. 
In effect, observe that because $M$ has time translation invariant distribution, then for every $h \in H$,  $M^{t}((s,r],A)(h) \overset{d}{=} M((s,r],A)(h)$, hence $\quadraVari{M^{t}(A)(h)}_{s}^{r} = \quadraVari{M(A)(h)}_{s}^{r}$. If $\nu^{t}_{h}$ denotes the intensity measure associated to $M^{t}$, the above arguments shows that  $\forall 0 \leq s \leq r$, $A \in \mathcal{A}$, $h \in H$, we have $\nu_{h}((s,r]\times A)=\nu^{t}_{h}((s,r]\times A)$ $\Prob$-a.e. Then by Theorem 3.11 in \cite{CCFM:SPDE} we have $ \nu_{h} = \nu^{t}_{h}$ for every $h \in H$. Then, by the definition of quadratic variation, we must have $\operQuadraVari{M^{t}}$ exists and indeed $\operQuadraVari{M} = \operQuadraVari{M^{t}} $.  
 Therefore stochastic integrals with respect to $M^{t}$ are well-defined. 

For any given $g \in G$, $h, t \geq 0$, it follows that 
\begin{flalign*}
& X(t+h,t;g) \\
& =  S(h)g + \int_{t}^{t+h}  S(t+h-s)B(X(s,t;g))  ds  + \int_{t}^{t+h} \int_U   S(t+h-s) F(u,X(s,t;g))  M(ds,du)  \\
& =  S(h)g + \int_{0}^{h} \, S(h-s)B(X(t+s,t;g)) \, ds + \int_{0}^{h} \int_U \,  S(h-s) F(u,X(t+s,t;g)) \, M^{t}(ds,du)  \\
& \overset{d}{=}  S(h)g + \int_{0}^{h} \, S(h-s)B(X(t+s,t;g)) \, ds + \int_{0}^{h} \int_U \,  S(h-s) F(u,X(t+s,t;g)) \, M(ds,du) 
\end{flalign*}
By uniqueness of solutions, $X(t,t+h;g)$ and $X(0,h;g)$ have the same distribution. 
\end{proof}

\begin{notation}
From now on we write $(X(t, \xi): t \geq 0)$ instead of $(X(t,0; \xi): t \geq 0)$,  $P_t$ instead of $P_{0,t}$, and $ P_{t}(g,\Gamma)$ instead of $P(0,g, t,\Gamma)$.  
\end{notation}

\begin{theorem}\label{theoTransitionSemigroIsFeller}     
$(P_t : t \geq 0)$ is a $C_0$-semigroup on $C_b(G)$. Moreover, for every $\varphi \in C_b(G)$ and $g\in G$ the mapping $t \mapsto (P_{t} \varphi)(g)$ is continuous.  
\end{theorem}
\begin{proof}
    First, we prove that $P_t(C_b(G)) \subseteq C_b(G)$.  Let $\varphi \in C_b(G)$, then
    $$
    \| P_t \varphi \|_{\infty} = \sup_{g \in G} |P_t\varphi(g)| \leq \Exp\left[\sup_{g \in G} |\varphi(X(t,0;g))| \right] < \infty,
    $$
    
    and we have the boundedness. To show continuity on $G$, by a standard density argument, we can assume that $\varphi$  is Lipschitz on $G$ with constant $K$ (see \cite{Heinonen}).   If $g_1, g_2 \in G$, we have
    \begin{align*}
        |P_t\varphi(g_1)-P_t\varphi(g_2)| & \leq \Exp[|\varphi(X(t,0;g_1)) - \varphi(X(t,0;g_2))|] \\ 
        & \leq K \Exp[\|X(t,0;g_1) - X(t,0;g_2)\|_G] 
    \end{align*}
    which together with Lemma \ref{lemmaContiDepenInitialData} prove the continuity. That $(P_{t}: t \geq 0)$ is a semigroup of linear operators on $C_{b}(G)$ is a direct consequence of Proposition \ref{propChapmanKolmogorov}. 

To finalize our proof we must show that for every $\varphi \in C_b(G)$ and $g\in G$ the mapping $t \mapsto (P_{t} \varphi)(g)$ is continuous. 

Let $0 \leq s , t \leq T$ and $g \in G$. Now, assume that $\varphi$  is Lipschitz on $G$ with constant $K$. We have 
    \begin{align*}
        \abs{ (P_t \varphi)(g) - (P_s \varphi)(g) } 
        & \leq  \Exp[ |\varphi(X(t,0;g)) - \varphi(X(s,0;g))| ] \\
        & \leq  K\Exp[\|X(t,0;g) - X(s,0;g)\|_G ]. 
    \end{align*}
    This and Lemma \ref{lemmaMeanSquareContinuity} prove  that the mapping $t \mapsto (P_t \varphi)(g)$ is continuous.  By a density argument one can show that $t \mapsto (P_t \varphi)(g)$ is continuous for $\varphi \in C_{b}(G)$ and $g \in G$.  
\end{proof}

\begin{remark}
If we combine the results of Theorems  \ref{theoHomogeMarkovProcess} and \ref{theoTransitionSemigroIsFeller} we conclude that $(X(t, \xi): t \geq 0)$ is a time homogeneous Markov process with transition function $(P_{t}(\xi,\Gamma): t\geq 0, \Gamma \in \mathcal{B}(G))$  and Feller transition semigroup $(P_{t}: t \geq 0)$.     
\end{remark}

\textbf{A note on the quadratic variation.} We finalize this section with a discussion on the implications of our assumptions for $M$ in the form taken by the quadratic variation $\operQuadraVari{M}$. 

In effect, as we have seen our hypothesis on $M$ of  independent increments and of time translation invariant distributions are essential to prove that the solution to the SPDEs with time homogeneous coefficients is a Feller process. Moreover, by these hypothesis for each $h \in H $ and $A \in \mathcal{A}$ we have $(M(t,A)(h): t \geq 0)$ is a real-valued zero-mean square integrable L\'evy martingale. Therefore, if we denote $\Exp \left[ \abs{M(1,A)(h)}^{2}\right]$ by $\sigma^{2}_{A,h}$, we must have 
\begin{equation} \label{eqDefinMeasureSigmaHA}
\Exp \left[ \abs{M(t,A)(h)}^{2}\right] = t \sigma^{2}_{A,h}
\end{equation}
This way, by the corresponding properties of $M$ as a orthogonal cylindrical martingale-valued measure, we can check that for any given $h \in H$, the set function $\mu_{h}: \mathcal{A} \rightarrow \R$ given by  $\mu_{h}(A) =\sigma^{2}_{A,h}$ defines a $\sigma$-finite measure on $\mathcal{A}$ which can be extended to $\mathcal{B}(U)$. Moreover, as a consequence of \eqref{eqDefinMeasureSigmaHA} we have the following equality for the family of intensity measures of $M$:
\begin{equation}\label{eqIntensiMeasureAsProductMeasure}
    \nu_{h}=\mu_{h} \otimes \mbox{Leb}. 
\end{equation}
Therefore, by Lemma \ref{lemmaProductMeasuresSSup}  we conclude that 
\begin{equation}\label{eqQuadraVariaIsAProductMeasure}
\operQuadraVari{M}= \mu \otimes  \mbox{Leb}, \quad \mbox{ where } \mu=\sup_{\norm{h}=1} \mu_{h}.
\end{equation}
In particular, we have that $\operQuadraVari{M}$ is a deterministic measure.

\begin{example}\label{examCMVMCylindriLevyProcesses}
Let $Z=(Z_{t}: t \geq 0)$ be a cylindrical zero-mean square integrable L\'{e}vy process in $X$, i.e. for every $d \in \N$, $h_{1}, \cdots, h_{d} \in H$ the $\R^{d}$-valued stochastic process $(Z_{t}(h_{1}), \cdots, Z_{t}(h_{d}): t \geq 0)$ is a  L\'{e}vy process in $\R^{d}$, and $\Exp [Z_{t}(h) ] = 0 $ and $\Exp [ \abs{Z_{t}(h)}^{2} ]< \infty$ for every $t \geq 0$, $h \in H$. We assume  that the mapping $Z_{t}: H \rightarrow L^{0} \ProbSpace$ is continuous. 

With the above hypothesis, it follows by Theorem 4.7 in \cite{ApplebaumRiedle:2010} that there exists a positive symmetric operator $Q:H \rightarrow H$, called the the \emph{covariance operator} of $Z$, defined by $\inner{Qh_{1}}{h_{2}}_{H}=\Exp \left[ Z_{1}(h_{1}) Z_{1}(h_{2}) \right] $ $\forall h_1 , h_2 \in H$. The operator $Q$ is linear and continuous by Proposition III.1.1 in \cite{VakhaniaTarieladzeChobanyan}. For every $h \in H$, observe that $\quadraVari{Z(h)}_{t} = t \Exp \left[\abs{Z_{1}(h)}^{2} \right]= t \inner{ Q(h)}{h}_{H}$. 

Consider a one point set $U=\{a\}$, $\mathcal{A}=2^{U}$ and let $M=(M(t,A): t \geq 0, A \in \mathcal{A})$ be defined by 
\begin{equation}\label{eqCMVMContiyliMartingale}
M(t,A)(h)=Z_{t}(h) \delta_{a}(a), \quad \, \forall h \in H, \, t \geq 0, A \in \mathcal{A}.
\end{equation}
We know, by Examples 5.14 and 5.22 in \cite{CCFM:SPDE}, that $M$ is a cylindrical orthogonal martingale-valued measure with family of intensity measures $\nu_h(ds,du)=\inner{Qh}{h}_H ds \delta_{a}(du) $, quadratic variation 
$\operQuadraVari{M}(\omega)(ds,du)=\norm{Q} ds \delta_{a}(du)$ and operator quadratic variation $ Q_{M}(\omega, r, u) = Q/\norm{Q}$. 

Let $Q_{Z}: \Omega \times [0,T] \rightarrow \mathcal{L}(H,H)$ be given by $Q_{Z}(\omega,t)=Q_{M}(\omega, t,a)$.  
We say that $\Phi:\Omega \times [0,T] \rightarrow \mathcal{HS}(H_{Q_{Z}},G)$ is in $\Lambda^{2}(Z,T)$, if $ (\omega,t) \mapsto \Phi(\omega,t) \circ Q^{1/2}_{Z}(\omega,t)(h)$ is  $\mathcal{P}_{T}/\mathcal{B}(G)$-measurable and satisfies
$$ \Exp \int_{0}^{T} \| \Phi(t)\circ Q_{Z}^{1/2}(t)\|^2_{\mathcal{HS}(H,G)} \, dt < \infty. $$
One can easily check from the definition of the stochastic integral in  \eqref{NewDefIntSimpleIntegrand}, that if $\Phi \in \Lambda^{2}(Z,T)$, then $\Phi \in \Lambda^{2}(M,T)$ and 
\begin{equation}\label{eqCompatibleintegrals}
\int_0^t \!\! \int_U \Phi(s)\, M(ds, du) = \int_0^t \Phi(s)\, dZ_{s},    
\end{equation} 
where the integrand in the right-hand side is the (radonifying) integral for cylindrical square integrable martingales as in \cite{MetivierPellaumail} (or as in \cite{DaPratoZabczyk:StochasticEquations, GawareckiMandrekar} for cylindrical Wiener processes). 
It is clear that $M$ has independent increments and a time translation invariant distribution. In comparison with \eqref{eqQuadraVariaIsAProductMeasure} we have  $\mu = \norm{Q} \delta_{a}(du)$.  

Let $A$, $B$ as in Section \ref{sectSolutionsSPDEs} with $B$ satisfying the growth and Lipschitz conditions in Assumption \ref{assumpLipschitz}. Let  $F= ( F(t,g): t\geq 0,  g \in G)$ such that  for all $ t\geq 0$, $g \in G$, $F(t,g) \in \mathcal{HS}(H, G)$, and for every $h \in H$, the mapping $(\omega, t,g) \mapsto F(t,u,g)\circ Q_{Z}^{1/2}(\omega, t)(h)$ is $\calP_{T}  \otimes \calB(G)/\calB(G)$-measurable. Moreover, assume the growth and Lipschitz conditions
\begin{align*}
    & \| F(t, g(t)) \circ Q^{1/2}_{Z} \|_{\mathcal{HS}(H,G)}^2   \leq C_F\left(1 +  \| g(t) \|^2  \right) \\
    & \|( F(t,  g(t)) - F(t, h(t)) )  \circ Q^{1/2}_{Z} \|_{\mathcal{HS}(H,G)}^2  \leq C_F \|  g(t) -h(t)\|^2 
\end{align*}
It is easy to see that the above conditions imply those in Assumption \ref{assumpLipschitz}. 

Consider the following SPDE
\begin{equation}\label{eqSPDECylLevy}
    dX_t = \left[ AX_t + B(t,X_t)\right]dt +  F(t,X_t) dZ_t.    
\end{equation}
Given our above formulation for $M$ and by \eqref{eqCompatibleintegrals} it is clear that \eqref{eqSPDECylLevy} is equivalent to \eqref{EqSPDE}. Hence, the SPDE \eqref{eqSPDECylLevy} has a unique mild (and weak) solution $( X_t: t \geq 0)$ which by Theorem \ref{theoremMarkov} is a Markov process with values in $G$. If we further assume that $B$ and $F$ are independent of the time variable $t$, then by Theorems  \ref{theoHomogeMarkovProcess} and \ref{theoTransitionSemigroIsFeller} we conclude that 
\begin{equation*}
  X_t = S(t)X_0 + \int_0^t S(t-s)B(X_s)ds + \int_0^t\!  S(t-s) F(X_s)dZ_s.      
\end{equation*}
is a time homogeneous Markov process with Feller transition semigroup.
\end{example}

\begin{remark}\label{remaCylinPoissonNoTimeInvariant}
Let $M$ as in Example \ref{examCylinPoissonIntegrals}. We have shown that $M$ has weakly independent increments. However, it might not have a time translation invariant distribution. To see this, observe that for $h \in H$, $0 \leq s<t$, $A \in \mathcal{A}$, we have
$$ \mathbb{E}|M(t,A)(h)-M(s,A)(h)|^2
=
\|h\|^2 \int_s^t \!\! \int_A |a(r,u)|^2\,dr\,\nu(du).$$
If $M$ has a time translation invariant distribution, this quantity would depend only on $t-s$. However, one can choose $a$ so that the above is not satisfied.  

For example, if $a(r,u)=r\,b(u)$ with $b\in L^2(U,\nu)$ nonzero, then
$$\mathbb{E}|X_t(x)-X_s(x)|^2
=
\|x\|^2 \|b\|_{L^2(U,\nu)}^2 \frac{t^3-s^3}{3},$$
which is not a function of $t-s$.
\end{remark}

\section{Time regularity for the solutions}\label{sectTimeRegularity}

In this section we prove that under additional conditions on the $C_{0}$-semigroup there exists a c\`adl\`ag version for the solution to \eqref{EqSPDE}. 

We start with the following Kotelenez type inequality for the stochastic convolution. Recall that $(S(t): t \geq 0)$ is called a \emph{quasi-contraction semigroup} if there exists $\theta \geq 0$ such that $\norm{S(t)}_{\mathcal{L}(G)}\leq e^{\theta t}$ for all $t \geq 0$. 

\begin{theorem}\label{theoKotelenez}
Assume $(S(t): t \geq 0)$  is a quasi-contraction semigroup. For any $F \in \Lambda^2(M, T)$, $C > 0$ and countable $D \subseteq [0, T]$,  
\begin{multline}\label{eqKotelenezInequality}
 \Prob \left( \sup_{t \in D} \norm{ \int_0^t \!\! \int_U S(t - r) \, F(r, u) \, M(dr, du) }_{G}^2 > C \right) \\ 
 \leq \frac{e^{2\theta T}}{C^2} \, \Exp \int_0^T \!\! \int_U \norm{ F(r, u) \, Q_M^{1/2} }_{\mathcal{HS} (H, G)}^2 \, \langle\!\langle M \rangle\!\rangle(dr, du).    
\end{multline}
\end{theorem}
\begin{proof}
We modify to our context the arguments used in the proof of Proposition 9.18 in \cite{PeszatZabczykSPDE}.

Let $Y(t) = \int_0^t \!\! \int_U S(t-r) F(r, u) \, M(dr, du)$. Consider a partition $0 = t_0 < t_1 < \dots < t_n = T$ and take $C > 0$. Then
\begin{align}
\Prob \left( \max_{1 \le k \le n} \| Y(t_k) \|_G > C \right) &= \sum_{k=1}^n \Prob \left( \bigcap_{j=1}^{k-1} \{ \| Y(t_j) \| \le C \} \cap \{ \| Y(t_k) \|_G > C \} \right) \nonumber \\
&\le \frac{1}{C^2} \sum_{k=1}^n I_{k}, \label{eqMaximalKotelenez}
\end{align}
where 
$$I_{k}= \Exp \left[ \| Y(t_k) \|_G^2 \, \caract_{\{ \| Y(t_k) \|_G > C \}} \cdot \caract_k \right], \quad  \mbox{and} \quad \caract_k = \prod_{j=1}^{k-1} \caract_{\{ \| Y(t_j) \| \leq C \}}. $$ 
Using the semigroup property of $(S(t): t \geq 0)$ and the action of linear continuous operators on the stochastic integral (see Proposition 6.15 in \cite{CCFM:SPDE}) we have
\begin{equation*}
Y(t_k) = S(t_k - t_{k-1}) Y(t_{k-1}) + \int_{t_{k-1}}^{t_k} \! \int_U S(t_k - r) \, F(r, u) \, M(dr, du). 
\end{equation*}
By the martingale property of the stochastic integral and the It\^o isometry \eqref{eqItoIsometrySimpleIntegrands},

\begin{flalign*}
&\widetilde{I}_{k} \defeq \Exp \left[ \| Y(t_k) \|_G^2 \, \caract_k \right] \\ &= \Exp \left[ \| S(t_k - t_{k-1}) Y(t_{k-1}) \|_G^2 \, \caract_k \right] 
 + \Exp \left[ \left\| \int_{t_{k-1}}^{t_k} \! \int_U S(t_k - r) \, F(r, u) \, M(dr, du) \right\|_G^2 \, \caract_k \right] \\
&\le e^{2\theta(t_k - t_{k-1})} \left\{ \Exp [ \| Y(t_{k-1}) \|_G^2 \, \caract_{k} ] + \Exp  \int_{t_{k-1}}^{t_k} \!  \int_U \| F(r, u) \, Q_M^{1/2} \|_{\mathcal{HS} (H, G)}^2 \, \operQuadraVari{M}(dr,du)  \right\}.
\end{flalign*}
Setting $\beta_{k}=e^{2\theta(t_k - t_{k-1})}$, 
$$J_{k-1}\defeq \Exp [ \| Y(t_{k-1}) \|_G^2 \, \caract_{k} ], \quad \alpha_k = \Exp   \int_{t_{k-1}}^{t_k} \!  \int_U \| F(r, u) \, Q_M^{1/2} \|_{\mathcal{HS} (H, G)}^2 \, \operQuadraVari{M}(dr,du)  ,$$
the above inequality can be written as
$$ 
\widetilde{I}_{k} \leq \beta_{k} \left(  J_{k-1} +   \alpha_{k}\right).
$$

Since $J_{k-1}+I_{k-1}=\widetilde{I}_{k-1} $, we have
\begin{equation}
\label{eqNobotar}
\widetilde{I}_{k} + \beta_k I_{k-1} \leq \beta_{k} ( \widetilde{I}_{k-1}+\alpha_{k}). 
\end{equation}
By induction we get
\begin{equation}
\label{eqIneqIk}
\widetilde{I}_k + \sum_{j=1}^{k} \beta_{k} \dots \beta_{k-j+1} I_{k-j} \leq \sum_{j=1}^{k} \beta_{k} \dots \beta_{k-j+1} \alpha_{k-j+1}.
\end{equation}
In fact, if \eqref{eqIneqIk} this is true for $k$ then by \eqref{eqNobotar} it follows that
$$
\widetilde{I}_{k+1} + \beta_{k+1} I_{k} \leq \beta_{k+1} \left( 
\sum_{j=1}^{k} \beta_{k} \dots \beta_{k-j+1} \alpha_{k-j+1} -
\sum_{j=1}^{k} \beta_{k} \dots \beta_{k-j+1} I_{k-j} + \alpha_{k+1}
\right)
$$
and rearranging terms we get \eqref{eqIneqIk} for $k+1$.

Since $1 \leq \beta_{k} \dots \beta_{k-j+1} \leq e^{2\theta T}$ for each $j$, \eqref{eqIneqIk} implies
$$
\widetilde{I}_k + \sum_{j=1}^{k} I_{k-j} \leq e^{2\theta T} \sum_{j=1}^{k} \alpha_{k-j+1}.
$$
Now observe that $I_k \leq \widetilde{I}_k$ and take $k=n$ to conclude
$$ 
\sum_{k=1}^{n} I_{k} \leq e^{2T \theta} \sum_{j=1}^{n} \alpha_{j} = e^{2\theta T} \, \Exp \int_0^T \!\! \int_U \| F(r, u) \, Q_M^{1/2} \|_{\mathcal{HS} (H, G)}^2 \, \langle\!\langle M \rangle\!\rangle(dr, du). 
$$

Combining the above inequality with \eqref{eqMaximalKotelenez} we get \eqref{eqKotelenezInequality}. 
\end{proof}

\begin{theorem}\label{theoStochConvoluCadlag}
Assume $(S(t): t \geq 0)$  is a quasi-contraction semigroup. For any $\Phi \in \Lambda^2(M, T)$, the stochastic convolution 
$$ X_{t}=\int_{0}^{t} \!\! \int_{U} \, S(t-r)\Phi(r,u) \, M(dr,du), \quad t \in [0,T],$$
has a square integrable, adapted, c\`adl\`ag version $(Y_{t}: t \in [0,T])$.  Moreover, if for each $A \in \mathcal{A}$ and $h \in H$, the real-valued process $(M(t,A)(h): t \geq 0)$ is continuous, then the results above remain valid replacing the property c\`{a}dl\`{a}g by continuous.
\end{theorem}
\begin{proof} We extend  the arguments used in the proof of Proposition 9.18 in \cite{PeszatZabczykSPDE}.
For $k \in \N$ let $s_k(r) =\sum_{i=0}^{2^{k-1}} \frac{iT}{2^k}\caract_{\left( \frac{iT}{2^k}, \frac{(i+1)T}{2^k} \right]}(r)$.
Define the $G$-valued process
\[
Y^k(t) = \int_0^T \!\! \int_U S(t - s_k(r)) \, \Phi(r, u) \, M(dr, du), \quad \forall t \in [0, T].
\]
Observe that $Y^{k}$ is c\`adl\`ag for each $k \in \N$. In fact, note that if $t \in \left( \frac{iT}{2^k}, \frac{(i+1)T}{2^k} \right]$ for some $0\leq i \leq 2^{k-1}$, we have 
\begin{align*}
Y^k(t) &= S\left(t - \tfrac{iT}{2^k} \right) \int_{\frac{iT}{2^k}}^t \int_U \Phi(r, u) \, M(dr, du) \\
&\quad + \sum_{j=1}^i S\left( t - \tfrac{(j-1)T}{2^k} \right) \int_{\frac{(j-1)T}{2^k}}^{\frac{jT}{2^k}} \int_U \Phi(r, u) \, M(dr, du).
\end{align*}
By the strong continuity of $(S(t): t \geq 0)$, and since the stochastic integral has c\`adl\`ag paths, we conclude that $Y^{k}$ is c\`adl\`ag on $ \left( \frac{iT}{2^k}, \frac{(i+1)T}{2^k} \right]$ for all $0\leq i \leq 2^{k-1}$, threfore it is c\`adl\`ag on $[0,T]$. 

Moreover, notice that
\begin{align*}
\Exp \int_0^T \!\! \int_U & \| (S(t - s_k(r)) - S(t - r)) \, \Phi(r, u) \, Q_M^{1/2} \|_{\mathcal{HS}(H, G)}^2 \, \langle\!\langle M \rangle\!\rangle(dr, du) \\
&\le 2 e^{2\theta t} \, \Exp \int_0^T \!\! \int_U \| \Phi(r, u) \, Q_M^{1/2} \|_{\mathcal{HS}(H, G)}^2 \, \langle\!\langle M \rangle\!\rangle(dr, du) < \infty.
\end{align*}
The strong continuity of $(S(t): t \geq 0)$ and dominated convergence imply that
\[
\lim_{k \to \infty} \Exp \int_0^T \!\! \int_U \| (S(t - s_k(r)) - S(t - r)) \, \Phi(r, u) \, Q_M^{1/2} \|_{\mathcal{HS}(H, G)}^2 \, \langle\!\langle M \rangle\!\rangle(dr, du) = 0.
\]
Then, by the It\^o isometry \eqref{eqItoIsometrySimpleIntegrands} we have
\[
\lim_{k \to \infty} \Exp \left[ \left\| \int_0^T \!\! \int_U (S(t - s_k(r)) - S(t - r)) \, \Phi(r, u) \, M(dr, du) \right\|_G^2 \right] = 0.
\]
Therefore, for every $t \in [0,T]$, $
Y^k(t)$ converges in $L^2(\Omega, \mathcal{F}, \Prob; G)$ to $X_{t}$ as $k \rightarrow \infty$. This way, since the space of c\`agl\`ag processes on $[0,T]$ is complete when equipped with the topology of uniform convergence in probability,  we can prove the existence of a c\`adl\`ag version of $(X_{t}: t \in [0,T])$ if we can show that $(Y^k: t\in [0,T])$ is Cauchy in probability uniformly on $[0,T]$.  To prove it, observe that for $m \geq k$, we have $s_m(r) \geq s_k(r)$, hence 
\begin{align}
Y^m(t) - Y^k(t) &= \int_0^T \!\! \int_U (S(t -s_m(r)) - S(t - s_k(r))) \, \Phi(r, u) \, M(dr, du) \label{eqCauchyCadlagVersion} \\
&= \int_0^T \!\! \int_U S(t - s_m(r)) (I - S(s_m(r) - s_k(r))) \, \Phi(r, u) \, M(dr, du). \nonumber
\end{align}

Let $F^{m,k}(w, r, u) = (I - S(s_m(r) - s_k(r))) \, \Phi(\omega, r, u)$, for $(\omega, r, u )\in \Omega \times [0,T]\times U$, and observe that  
\begin{align*}
\Exp \int_0^T\!\! \int_U & \norm{  F^{m,k}(r, u) \, Q_M^{1/2} }_{\mathcal{HS}(H, G)}^2 \, \langle\!\langle M \rangle\!\rangle(dr, du) \\
&\le (1 + e^{2\theta T}) \, \Exp \int_0^T \!\!  \int_U \norm{ \Phi(r, u) \, Q_M^{1/2} }_{\mathcal{HS}(H, G)}^2 \, \langle\!\langle M \rangle\!\rangle(dr, du) < \infty.
\end{align*}

Hence $F^{m,k} \in \Lambda^2(M, T)$. By \eqref{eqKotelenezInequality}, \eqref{eqCauchyCadlagVersion},  and the fact that both $Y^{m}$ and $Y^{k}$ are c\`adl\`ag, for any $C>0$ we have
\begin{equation*}
\Prob \left( \sup_{0 \le t \le T} \| Y^m(t) - Y^k(t) \|_G^2 > C \right) 
\leq \frac{e^{2\theta T}}{C^2} \, \Exp \int_0^T \!\! \int_U \norm{ F^{m,k}(r, u) \, Q_M^{1/2} }_{\mathcal{HS}(H, G)}^2 \, \langle\!\langle M \rangle\!\rangle(dr, du).
\end{equation*}
Now, the strong continuity of $(S(t):t \geq 0)$ shows that for every $(\omega, r, u )\in \Omega \times [0,T]\times U$, we have 
\[
\lim_{m, k \to \infty} \| F^{m,k}(\omega, r, u) \, Q_M^{1/2}(\omega, r, u) \|_{\mathcal{HS}(H, G)}^2 = 0.
\]
Therefore, by the dominated convergence theorem, we conclude that
\[
\lim_{m, k \to \infty} \Prob \left( \sup_{0 \le t \le T} \| Y^m(t) - Y^k(t) \|_G^2 > C \right) = 0.
\]
This shows that $(Y^k: t\in [0,T])$ is Cauchy in probability uniformly on $[0,T]$, and so there exists a square integrable, adapted, c\`adl\`ag version $(Y_{t}: t \in [0,T])$ of $(X_{t}: t \in [0,T])$.   

Finally, if each $(M(t,A)(h): t \geq 0)$ has continuous paths, then the stochastic integrals defined with respect to $M$ also have continuous paths (Proposition 6.14 in \cite{CCFM:SPDE}). Hence, each $Y^{k}$ has continuous paths. Since the space of continuous processes on $[0,T]$ is closed when equipped with the topology of uniform convergence in probability, from the arguments above we conclude that $(Y_{t}: t \in [0,T])$ is a continuous paths version of $(X_{t}: t \in [0,T])$. 
\end{proof}

\begin{lemma}\label{lemBochnerConvoluContinuous}
Let $(S(t): t \geq 0)$  be a quasi-contraction semigroup on $G$, and $\Psi: \Omega \times \R_+  \to G$ satisfies $\Exp \int_0^T \| \Psi(s) \|_G \, ds < \infty$, then the convolution 
$$ \int_{0}^{t} S(t-r) \Psi(s) \, ds, \quad t \in [0,T],$$
has continuous paths $\mathbb{P}$-a.e. 
\end{lemma}

\begin{proof}
    Since $\Exp \int_0^T \| \Psi(s) \|_G \, ds < \infty$, Fubini's theorem implies that
$\Psi(\cdot,\omega) \in L^1([0,T];G)$ for $\mathbb{P}$-a.e. $\omega$. Therefore, the result follows from the fact that the convolution of a strongly continuous function with an integrable function is continuous. The proof in the Bochner-valued setting is identical to the classical scalar-valued case (for a more general result, see Lemma 1 in \cite{DaPratoKwapienZabczyk:1988}).
\end{proof}

Finally, we show the existence of c\`adl\`g version for the solution to \eqref{EqSPDE}. 

\begin{theorem}\label{theoExistenceCadlagVersion}
Assume $(S(t): t \geq 0)$  is a quasi-contraction semigroup with infinitesimal generator $A$. Let $B$ and $F$ be as in Section \ref{sectSolutionsSPDEs}. Then \eqref{EqSPDE} has a unique c\`adl\`ag  weak solution with initial condition $X_0$.   Moreover, if for each $A \in \mathcal{A}$ and $h \in H$, the real-valued process $(M(t,A)(h): t \geq 0)$ is continuous, then the results above remain valid replacing the property c\`{a}dl\`{a}g by continuous.
\end{theorem}
\begin{proof}
We already know that the unique weak solution to \eqref{EqSPDE} is its mild solution \eqref{EqStochIntDiffEq}. Hence, it suffices to show that its mild solution has a c\`adl\`ag version. In fact, by the strong continuity of $(S(t):t \geq 0)$  the process $(S(t)X_0: t \geq 0)$ has continuous paths. Now, since $\Exp \int_{0}^{T} \norm{B(s,X_s)}_{G}ds < \infty$, by Lemma \ref{lemBochnerConvoluContinuous} we have $\int_0^t S(t-s)B(s,X_s)ds$ has continuous paths $\Prob$-a.e. Finally, since $(F(t,u,X_t(\omega))) \in \Lambda^{2}(M,T)$ it follows from Theorem \ref{theoStochConvoluCadlag} that the process $\int_0^t\! \! \int_U S(t-s) F(s,u,X_s)M(ds,du)$  has a c\`adl\`ag version (continuous version if each $(M(t,A)(h): t \geq 0)$ has continuous paths). Hence the (mild) solution $(X_{t}: t \in [0,T])$ has a  c\`adl\`ag version (continuous version if each $(M(t,A)(h): t \geq 0)$ has continuous paths).
\end{proof}

\begin{example}
In the setting of Example \ref{examCylinPoissonIntegrals}, the heat semigroup $(S(t): t \geq 0)$ is a $C_{0}$-semigroup of contractions (thus $\theta=0$). Hence, the mild solution $(X_{t}: t \geq 0)$ has a c\`adl\`ag version according to Theorem \ref{theoExistenceCadlagVersion}. 
\end{example}

\begin{example}
In the setting of Example \ref{examCMVMCylindriLevyProcesses}, assume $(S(t): t \geq 0)$ is a quasi-contraction semigroup. By Theorem \ref{theoExistenceCadlagVersion} the mild solution $(X_{t}: t \geq 0)$ to \eqref{eqSPDECylLevy} has a c\`adl\`ag version. Moreover, if $(Z_{t}: t \geq 0)$ is a cylindrical Wiener process, then $(X_{t}: t \geq 0)$ has a continuous version. 
\end{example}

\appendix
\section{A note on the supremum of measures}

In the following result the supremum is taken in the sense of supremum of measures as defined in Section  2 of \cite{CCFM:SPDE}. 

\begin{lemma}\label{lemmaProductMeasuresSSup}
Given a family of measures $(\lambda_j)_{j\in J}$ defined on the space $(\Omega_1,\calF_1)$ and a measure $\lambda$ defined on $(\Omega_2,\calF_2)$, we have
$$
\sup_{j\in J} (\lambda_j\otimes\lambda) = (\sup_{j\in J}\lambda_j) \otimes \lambda
$$
\end{lemma}
\begin{proof}
In fact, let $\mu := \sup_{j\in J}\lambda_j$. It is clear that $\mu \otimes \lambda \geq \lambda_j\otimes\lambda$ for each $j\in J$. On the other hand, if $\nu$ is a measure on $(\Omega_1\times\Omega_2,\calF_1\otimes\calF_2)$ that dominates each $\lambda_j\otimes\lambda$, by fixing $B\in \mathcal{F}_2$ we obtain that the measure $\nu(\cdot\times B)$ dominates each $\lambda(B) \lambda_j$, so
$$
\nu(\cdot \times B) \geq \sup_{j\in J} \lambda(B) \lambda_j = \lambda(B)\mu.
$$
This shows that $\nu$ dominates $\mu\otimes\lambda$ on rectangles, so $\nu \geq \mu \otimes \lambda$.
\end{proof}


\begin{thebibliography}{HD}



\bibitem{AlvaradoFonseca:2021} Alvarado-Solano, A. E.; Fonseca-Mora, C. A.: Stochastic integration in Hilbert spaces with respect to cylindrical martingale-valued measures, \emph{ALEA Lat. Am. J. Probab. Math. Stat. } 18, no. 2, 1267--1295 (2021).


\bibitem{ApplebaumRiedle:2010} Applebaum, D.; Riedle, M.: Cylindrical Lévy processes in Banach spaces,  \emph{Proc. Lond. Math. Soc.}, (3) 101, no. 3, 697--726 (2010).



\bibitem{CCFM:SPDE} Cambronero, S.; Campos, D.; Fonseca-Mora, C.A.; Mena, D: Cylindrical Martingale-Valued Measures, Stochastic Integration and SPDEs, \emph{Stoch PDE: Anal Comp}, 13,no.2, 887--955 (2025). 




\bibitem{DaPratoZabczyk:StochasticEquations}  Da~Prato, G.; Zabczyk, J.:  {\it Stochastic equations in infinite dimensions}, second edition, 
Encyclopedia of Mathematics and its Applications, 152, Cambridge Univ. Press, Cambridge (2014). 

\bibitem{DaPratoKwapienZabczyk:1988} Da Prato, G.; Kwapie\v{n}, S.; Zabczyk, J.: Regularity of solutions of linear stochastic equations in hilbert spaces, \emph{Stochastics}, 23(1), 1–23 (1988). 

\bibitem{GawareckiMandrekar} Gawarecki, L.;  Mandrekar, V.~S.: {\it Stochastic differential equations in infinite dimensions with applications to stochastic partial differential equations}, Probability and its Applications (New York), Springer, Heidelberg (2011).


\bibitem{Heinonen} Heinonen, J.: \emph{Lectures on analysis on metric spaces}, Universitext Springer-Verlag, New York (2001). 


\bibitem{Kotelenez:1982} Kotelenez, P.:
A submartingale type inequality with applications to stochastic evolution equations, \emph{Stochastics}, {8}(2): 139--151 (1982). 

\bibitem{KumarRiedle} Kumar, U.; Riedle, M.: The stochastic Cauchy problem driven by a cylindrical L\'evy process, \emph{Electron. J. Probab.}, {\bf 25}, Paper No. 10, 26 pp. (2020). 

\bibitem{KumarRiedleInvariant} Kumar, U.; Riedle, M.:  Invariant measure for the stochastic Cauchy problem driven by a cylindrical L\'evy process, \emph{J. Math. Anal. Appl.}, {\bf 493}, no.~2, Paper No. 124536, 26 pp. (2021).

\bibitem{LiuZhai} Liu,Y.; Zhai,J.:  Time regularity of generalized Ornstein-Uhlenbeck processes with L\'evy noises in Hilbert spaces, \emph{J. Theoret. Probab.}, {\bf 29}, no.~3, 843--866 (2016). 

 
\bibitem{LiuZhainote} Liu,Y.; Zhai,J.: A note on time regularity of generalized Ornstein-Uhlenbeck processes with cylindrical stable noise, \emph{C. R. Math. Acad. Sci. Paris}, {\bf 350}, no.~1-2, 97--100 (2012). 


\bibitem{MetivierPellaumail} M\'{e}tivier, M; Pellaumail, J.:  \emph{Stochastic integration}. Probability and Mathematical Statistics, Academic Press, New York (1980).

\bibitem{Ondrejat:2005} Ondrej\'at, M.: Brownian representations of cylindrical local martingales,
martingale problem and strong Markov property of weak solutions of SPDEs in Banach spaces, \emph{Czechoslovak Math. J.}, 55(130), no.4, 1003--1039 (2005). 


 \bibitem{PeszatZabczykSPDE} Peszat, S.; Zabczyk, J.: \emph{Stochastic partial differential equations with Lévy noise. An evolution equation approach}. Encyclopedia of Mathematics and its Applications, 113, Cambridge University Press, Cambridge (2007).

 



 \bibitem{PeszatZabczykregularity} Peszat, S.; Zabczyk, J.: Time regularity of solutions to linear equations with L\'evy noise in infinite dimensions, \emph{Stochastic Process. Appl.}, {\bf 123}, no.~3, 719--751 (2013). 

\bibitem{VakhaniaTarieladzeChobanyan} Vakhania, N. N.; Tarieladze, V. I.; Chobanyan, S. A.: \emph{Probability distributions on Banach spaces}.  Mathematics and its Applications (Soviet Series), 14. D. Reidel Publishing Co., Dordrecht (1987).


 \bibitem{Walsh:1986} Walsh, John B.: \emph{An introduction to stochastic partial differential equations}. \'{E}cole d'\'{e}t\'{e} de probabilit\'{e}s de Saint-Flour, XIV—1984, 265–439, Lecture Notes in Math., 1180, Springer, Berlin (1986).

\end{thebibliography}
\end{document}